\documentclass[12pt,reqno,a4paper]{amsart}

\newcommand\version{February 19, 2025}

\usepackage{amsmath,amsfonts,amsthm,amssymb,amsxtra}
\usepackage{hyperref}
\usepackage{bbm} 
\usepackage{bm} 
\usepackage{accents} 

\newtheorem{theorem}{Theorem}[section]
\newtheorem{proposition}[theorem]{Proposition}

\theoremstyle{definition}

\theoremstyle{remark}

\newtheorem{remarks}[theorem]{Remarks}

\numberwithin{equation}{section}

\newcommand{\C}{\mathbb{C}}

\newcommand{\D}{\mathcal{D}}

\renewcommand{\epsilon}{\varepsilon}

\renewcommand{\phi}{\varphi}
\newcommand{\R}{\mathbb{R}}

\DeclareMathOperator{\re}{Re}

\def\bs{\mathbb{S}}

\def\cl{\mathcal{L}}

\def\cq{\mathcal{Q}}

\def\Rd{{\mathbb{R}^d}}

\newcommand{\me}[1]{\mathrm{e}^{#1}}
\newcommand{\one}{\mathbf{1}}

\makeatletter
\newcommand*{\rom}[1]{\expandafter\@slowromancap\romannumeral #1@}
\makeatother

\begin{document}

\title[Subordinated Bessel heat kernels --- \version]{Subordinated Bessel heat kernels}

\author[K. Bogdan]{Krzysztof Bogdan}
\address[Krzysztof Bogdan]{Department of Pure and Applied Mathematics, Wroc\l aw University of Science and Technology, Hoene-Wro\'nskiego 13C, 50-376 Wroc\l aw, Poland}
\email{krzysztof.bogdan@pwr.edu.pl}

\author[K. Merz]{Konstantin Merz}
\address[Konstantin Merz]{Institut f\"ur Analysis und Algebra, Technische Universit\"at Braunschweig, Universit\"atsplatz 2, 38106 Braun\-schweig, Germany, and Department of Mathematics, Graduate School of Science, Osaka University, Toyonaka, Osaka 560-0043, Japan}
\email{k.merz@tu-bs.de}

\thanks{K.B. was supported through the DFG-NCN Beethoven Classic 3 programme, contract no. 2018/31/G/ST1/02252 (National Science Center, Poland) and SCHI-419/11–1 (DFG, Germany).
  K.M. was supported through the PRIME programme of the German Academic Exchange Service (DAAD) with funds from the German Federal Ministry of Education and Research (BMBF)}

\date{\version}

\begin{abstract}
    We prove new bounds for Bessel heat kernels and Bessel heat kernels subordinated by stable subordinators. In particular, we provide 3G inequalities in the subordinated case.
\end{abstract}

\maketitle

\section{Introduction and main result}

\label{s:semigroupproperties}

The Bessel differential operators
$-\tfrac{d^2}{dr^2}-\tfrac{2\zeta}{r}\,\tfrac{d}{dr}$ 
are of fundamental importance in mathematics, physics, finance, and engineering. They are used to describe heat and wave propagation under spherical symmetry.
The corresponding Bessel heat kernels are crucial for harmonic analysis and stochastic processes related to the Brownian motion, see, e.g., Revuz and Yor \cite{RevuzYor1999}. Our goal is to present estimates for the $\alpha/2$-subordinated Bessel heat kernels $p_\zeta^{(\alpha)}$, $\alpha\in(0,2]$. We are motivated by the appearance of $p_\zeta^{(\alpha)}$ in the ground state representation given in our work \cite{BogdanMerz2024}. In fact, the estimates are essential in our recent and forthcoming papers \cite{Bogdanetal2024,BogdanMerz2024H} concerning Schr\"odinger perturbations of $p_\zeta^{(\alpha)}(t,r,s)$ by Hardy potentials, but since they are of independent wide interest, we present them in this separate note.
For recent studies on the generator of $p_\zeta^{(\alpha)}$, which is just the $\tfrac{\alpha}{2}$-th fractional power of the Bessel operator \cite{BogdanMerz2024,BouzeffourGarayev2022}, we refer, e.g., to Betancor, Castro, and Stinga \cite{Betancoretal2014}, Galli, Molina, and Quintero \cite{Gallietal2022}, and Bouzeffour and Garayev \cite{BouzeffourGarayev2022}.

\smallskip
We now introduce our setting. For $\zeta\in(-1/2,\infty)$, we define the Bessel heat kernel
\begin{align}
  \label{eq:defpheatalpha2}
  \begin{split}
    p_\zeta^{(2)}(t,r,s) & : = \frac{(rs)^{1/2-\zeta}}{2t}\exp\left(-\frac{r^2+s^2}{4t}\right)I_{\zeta-1/2}\left(\frac{rs}{2t}\right), \quad r,s,t>0.
  \end{split}
\end{align}\index{$p_\zeta^{(2)}(t,r,s)$}Here and below, for $z\in\C\setminus(-\infty,0]$, $I_\nu(z)$ denotes the modified Bessel function of the first kind of order $\nu\in\C$ \cite[(10.25.2)]{NIST:DLMF}.\index{$I_\nu$}

The kernel $p_\zeta^{(2)}(t,r,s)$ with the reference (speed) measure $r^{2\zeta}dr$ on $\R_+$ is the transition density of the Bessel process of order $\zeta-1/2$ \emph{reflected at the origin}.
We remark that the Bessel process of order $\zeta-1/2$ \emph{killed at the origin} has the transition density given in \eqref{eq:defpheatalpha2} with $I_{\zeta-1/2}(\cdot)$ replaced with $I_{|\zeta-1/2|}(\cdot)$, see Borodin and Salminen \cite[Appendix~1.21, p.~133--134]{BorodinSalminen2002}. The distinction is superfluous when $\zeta\geq1/2$ because on the one hand $|\zeta-1/2|=\zeta-1/2$ and on the other hand the Bessel process does not hit the origin, so no conditions (reflecting or killing) are to be imposed at the origin. See also Malecki, Serafin, and {\.Z}{\'o}rawik \cite{Maleckietal2016}. Recall that $p_{(d-1)/2}^{(2)}(t,r,s)$ is the transition density of the radial part of the Brownian motion in $\R^d$ with the clock $2t$.
For further information on $p_\zeta^{(2)}$,
we refer, e.g., to the textbooks \cite[Part I, Section~IV.6 or Appendix~1.21]{BorodinSalminen2002} or \cite[Chapter~XI]{RevuzYor1999}.

\smallskip
We now define the $\tfrac\alpha2$-subordinated Bessel heat kernels for $\alpha\in(0,2)$. Recall that for $\alpha\in(0,2)$ and $t>0$, by Bernstein's theorem, the completely monotone function $[0,\infty)\ni\lambda\mapsto\me{-t\lambda^{\alpha/2}}$  is the Laplace transform of a probability density function $\R_+\ni\tau\mapsto\sigma_t^{(\alpha/2)}(\tau)$. Thus,
\begin{align}
  \label{eq:subordination}
  \me{-t\lambda^{\alpha/2}} = \int_0^\infty \me{-\tau\lambda}\,\sigma_t^{(\alpha/2)}(\tau)\,d\tau, \quad t>0,\,\lambda\geq0,
\end{align}
see, e.g., Schilling, Song, and Vondra\v{c}ek \cite[(1.4) and Chapter~5]{Schillingetal2012}.
In \cite[Appendix~B]{BogdanMerz2024} we list some useful properties of $\sigma_t^{(\alpha/2)}(\tau)$, sharp estimates for $\sigma_t^{(\alpha/2)}(\tau)$, and references. We define the $\tfrac\alpha2$-subordinated Bessel heat kernel with the reference measure $r^{2\zeta}dr$ on $\R_+$ as
\begin{align}
  \label{eq:defpheatalpha}
  \begin{split}
    p_\zeta^{(\alpha)}(t,r,s) & : = \int_0^\infty p_\zeta^{(2)}(\tau,r,s)\,\sigma_t^{(\alpha/2)}(\tau)\,d\tau, \quad r,s,t>0.
  \end{split}
\end{align}\index{$p_\zeta^{(\alpha)}(t,r,s)$}
We should note that more general subordination and references are discussed in Grzywny and Trojan \cite{GrzywnyTrojan2021} and \cite{Schillingetal2012}, but our main motivation for this study is the fact that $p_\zeta^{(\alpha)}$ arises when considering the $d$-dimensional fractional Laplacian $(-\Delta)^{\alpha/2}$ on the space of multiples of solid harmonics, i.e., functions of the form $[u]_{\ell,m}(x):=u(|x|)|x|^\ell Y_{\ell,m}(x/|x|)$, where $Y_{\ell,m}$ is a $L^2(\bs^{d-1})$-normalized spherical harmonic, $u\in L^2(\R_+,r^{2\zeta}dr)$, and $\zeta=(d-1+2\ell)/2$. Namely,
\begin{align}
  \label{eq:linksemigroupsrdrplus}
  \begin{split}
    \langle[u]_{\ell,m},\me{-t(-\Delta)^{\alpha/2}}(t,\cdot,\cdot)[u]_{\ell,m}\rangle_{L^2(\R^d)}
    & = \langle u, p_{(d-1+2\ell)/2}^{(\alpha)}(t,\cdot,\cdot)u\rangle_{L^2(\R_+,r^{2\zeta}dr)},
      \quad t>0.
  \end{split}
\end{align}
Furthermore, the following equality holds pointwise,
\begin{align}
  & p_{\frac{d-1+2\ell}{2}}^{(\alpha)}(t,r,s)
    = \iint\limits_{\bs^{d-1}\times\bs^{d-1}} \overline{[1]_{\ell,m}(r\omega_x)} [1]_{\ell,m}(s\omega_y) \me{-t(-\Delta)^{\alpha/2}}(t,r\omega_x,s\omega_y)\,d\omega_x\,d\omega_y, \quad r,s,t>0,
\end{align}
where $\me{-t(-\Delta)^{\alpha/2}}$ is the heat kernel of $(-\Delta)^{\alpha/2}$ on $\R^d$. See \cite{BogdanMerz2024} for details.

\subsection*{Organization and notation}
\label{s:organization}

In Section~\ref{s:elementaryproperties}, we recall and prove sharp upper and lower bounds for $p_\zeta^{(\alpha)}(t,r,s)$ (Theorem~\ref{heatkernelalpha1subordinatedboundsfinal}), discuss $p_\zeta^{(\alpha)}(t,r,s)$ as the probability transition density and the kernel of a strongly continuous contraction semigroup on $L^2(\R_+,r^{2\zeta}dr)$ (Proposition~\ref{strongcontinuity}), and recall an explicit expression for $p_\zeta^{(1)}(t,r,s)$ in \eqref{eq:heatkernell}.
In Section~\ref{s:3g}, we prove a 3G inequality when $\alpha\in(0,2)$ (Theorem~\ref{3gheatalpha}).
In Section~\ref{s:comparison}, we prove further pointwise bounds, called comparability results (Theorem~\ref{comparablealpha}).
The technical part of the proof of Theorem~\ref{heatkernelalpha1subordinatedboundsfinal}, when $\alpha\in (0,2)$, is given in Appendix~\ref{s:proofheatkernelalpha1subordinatedboundsfinal}.

\medskip
Below we denote generic constants, i.e., numbers in $(0,\infty)$ by $c$. The values of constants may change from place to place. We may mark the dependence of $c$ on some parameter $\tau$ by the notation $c_\tau$ or $c(\tau)$. For functions $f,g\geq0$, we write $f\lesssim g$ to indicate that there is a constant $c$ such that $f\leq c g$. If $c$ depends on $\tau$, we may write $f\lesssim_\tau g$.
The notation $f\sim g$ means that $f\lesssim g\lesssim f$; we say \emph{$f$ is comparable to $g$}.\index{$\lesssim$}\index{$\sim$} We abbreviate $a\wedge b:=\min\{a,b\}$ and $a\vee b:=\max\{a,b\}$.\index{$\wedge$}\index{$\vee$}
The regularized hypergeometric function \cite[(15.2.1)]{NIST:DLMF} is denoted by $_2\tilde F_1(a,b;c;z) :=\, _2F_1(a,b;c;z)/\Gamma(c)$\index{$_2\tilde F_1(a,b;c;z)$}, with $a,b,c\in\C$ and $z\in\{w\in\C:\,|w|<1\}$.
We introduce further notation as we proceed.

\subsection*{Acknowledgments.} We thank Volker Bach, Kamil Bogus, Jacek Dziubański, Tomasz Grzywny, Tomasz Jakubowski, Jacek Ma\l{}ecki, Haruya Mizutani, Adam Nowak, Marcin Preisner, and Grzegorz Serafin for discussion and references.

\section{Fundamental properties, bounds, and explicit expressions}
\label{s:elementaryproperties}

We recall the following properties of $p_\zeta^{(\alpha)}(t,r,s)$ for $\zeta\in(-1/2,\infty)$, $\alpha\in(0,2]$ proved in \cite[Section~2]{BogdanMerz2024}. For all $t,t',r,s>0$, we have $p_\zeta^{(\alpha)}(t,r,s)=p_\zeta^{(\alpha)}(t,s,r)>0$,
\begin{align}
  \label{eq:normalizedalpha}
  & \int_0^\infty p_\zeta^{(\alpha)}(t,r,s) s^{2\zeta}\,ds = 1, \\
  \label{eq:chapman}
  & \int_0^\infty p_\zeta^{(\alpha)}(t,r,z) p_\zeta^{(\alpha)}(t',z,s) z^{2\zeta}\,dz = p_\zeta^{(\alpha)}(t+t',r,s), \quad \text{and} \\
  \label{eq:scalingalpha}
  & p_\zeta^{(\alpha)}(t,r,s) = t^{-\frac{2\zeta+1}{\alpha}} p_\zeta^{(\alpha)}\left(1,\frac{r}{t^{1/\alpha}},\frac{s}{t^{1/\alpha}}\right).
\end{align}

Using the L\'evy distribution
\begin{align}
  \label{eq:subordinatorexplicitalpha1}
  \sigma_t^{(1/2)}(\tau)
  = \frac{1}{2\sqrt\pi}\cdot \frac{t}{\tau^{3/2}}\me{-t^2/(4\tau)},
  \quad t,\tau>0,
\end{align}
see, e.g., Stein and Weiss \cite[p.~6]{SteinWeiss1971}, it is possible to give an explicit expression for $p_\zeta^{(\alpha)}(t,r,s)$ in the physically important case $\alpha=1$. One obtains, for $\zeta\in(-1/2,\infty)$ and $r,s,t>0$,
\begin{align}
  \label{eq:heatkernell}
  \begin{split}
    p_\zeta^{(1)}(t,r,s)
    & = \frac{2\Gamma(\zeta+1)}{\sqrt\pi} \cdot \frac{t}{\left(r^2+s^2+t^2\right)^{\zeta+1}} \\
      & \quad \times\, _2\tilde{F}_1\left(\frac{\zeta+1}{2},\frac{\zeta+2}{2};\zeta+\frac{1}{2};\frac{4r^2s^2}{\left(r^2+s^2+t^2\right)^2}\right),
  \end{split}
\end{align}
see, e.g., Betancor, Harboure, Nowak, and Viviani \cite[p.~136]{Betancoretal2010} for a computation. In particular,
\begin{subequations}
  \label{eq:heatkernell0}
  \begin{align}
    p_0^{(1)}(t,r,s) & = \frac{2 t}{\pi  \left(r^2+s^2+t^2\right) \left(1- 4 r^2 s^2 \cdot \left(r^2+s^2+t^2\right)^{-2}\right)}, \\
    p_1^{(1)}(t,r,s) & = \frac{4}{\pi}\, \frac{t}{(r^2-s^2)^2+t^2(t^2+2r^2+2s^2)}.
  \end{align}
\end{subequations}
Since explicit expressions for the $\alpha/2$-stable subordination density with rational $\alpha/2$ (see, e.g., Penson and G\'orska \cite{PensonGorska2010}) are available, one could also compute $p_\zeta^{(\alpha)}(t,r,s)$ for such $\alpha$, but  
the resulting expressions are rather involved when $\alpha\neq1$.
Using bounds for hypergeometric functions, one obtains the following upper and lower bounds 
\begin{align}
  p_\zeta^{(1)}(t,r,s) \sim_\zeta \frac{t}{(r^2+s^2+t^2)^\zeta [(r-s)^2+t^2]},
\end{align}
see, e.g., \cite[Proposition~6.1]{Betancoretal2010}.
In the following theorem, we give sharp upper and lower bounds for $p_\zeta^{(\alpha)}(t,r,s)$ and all $\alpha\in(0,2)$.
We remark that the case of $\alpha=1$ in Theorem~\ref{heatkernelalpha1subordinatedboundsfinal} is resolved in \cite[Proposition~6.1]{Betancoretal2010}; see also Dziuba\'{n}ski and Preisner \cite[Proposition~6]{DziubanskiPreisner2018} and \cite{Betancoretal2014}. The case of general $\alpha\in (0,2)$ seems unknown, although similar estimates were obtained by analogous techniques in various settings, see, e.g., Bogdan, St\'{o}s, and Sztonyk \cite[Theorem~3.1]{Bogdanetal2003H} for $\Rd$ and unbounded fractals; see also remarks in the proof below.

\begin{theorem}
  \label{heatkernelalpha1subordinatedboundsfinal}
  Let $\zeta\in(-1/2,\infty)$.
  Then, there are $c,c'>0$ such that
  \begin{subequations}
    \label{eq:easybounds2}
    \begin{align}
      p_\zeta^{(2)}(t,r,s)
      \label{eq:easybounds2a}
      & \asymp_\zeta t^{-\frac12}\frac{\exp\left(-\frac{(r-s)^2}{c t}\right)}{(rs+t)^{\zeta}} \\
      \label{eq:easybounds2b}
      & \asymp_\zeta \left(1\wedge\frac{r}{t^{1/2}}\right)^{\zeta}\, \left(1\wedge\frac{s}{t^{1/2}}\right)^{\zeta}\,\left(\frac{1}{rs}\right)^{\zeta} \cdot t^{-\frac12} \cdot \exp\left(-\frac{(r-s)^2}{c' t}\right),
    \end{align}
  \end{subequations}\index{$\asymp$}for all $r,s,t>0$. Moreover, for all $\alpha\in(0,2)$ and all $r,s,t>0$,
  \begin{align}
    \label{eq:heatkernelalpha1weightedsubordinatedboundsfinal}
    \begin{split}
      p_\zeta^{(\alpha)}(t,r,s)
      & \sim_{\zeta,\alpha} \frac{t}{|r-s|^{1+\alpha}(r+s)^{2\zeta} + t^{\frac{1+\alpha}{\alpha}}(t^{\frac1\alpha}+r+s)^{2\zeta}}.
    \end{split}
  \end{align}
\end{theorem}

Here and below the notation $\asymp_\zeta$ combines an upper bound and a lower bound similarly as $\sim_\zeta$, but the displayed constants in exponential factors (i.e., the constant $c$ in \eqref{eq:easybounds2a} and $c'$ in \eqref{eq:easybounds2b}) may be different in the upper and the lower bounds. Furthermore, as suggested by the notation $\sim_\zeta$, we allow the constants in the exponential factors to depend on $\zeta$, too. Thus, for instance, \eqref{eq:easybounds2a} is equivalent to the statement that there are $c_{j,\zeta}$, $j\in\{1,2,3,4\}$ such that
\begin{align}
  \begin{split}
    c_{1,\zeta} t^{-\frac12}\frac{\exp\left(-\frac{(r-s)^2}{c_{2,\zeta} t}\right)}{(rs+t)^{\zeta}}
    \leq p_\zeta^{(2)}(t,r,s)
    \leq c_{3,\zeta} t^{-\frac12}\frac{\exp\left(-\frac{(r-s)^2}{c_{4,\zeta} t}\right)}{(rs+t)^{\zeta}}.
  \end{split}
\end{align}

\begin{proof}[Proof of Theorem~\ref{heatkernelalpha1subordinatedboundsfinal}]
  For $\alpha=2$, the two-sided estimates in \eqref{eq:easybounds2a} follow from the following asymptotics \cite[(10.30.1), (10.30.4)]{NIST:DLMF},
  \begin{align}
    I_\rho(z)
    \sim \frac{1}{\Gamma(\rho+1)}\cdot \left(\frac{z}{2}\right)^\rho \one_{z\leq1} + \frac{\me{z}}{\sqrt{2\pi z}}\one_{z\geq1}, \quad z\geq0,\,\rho\notin\{-1,-2,...\}.
  \end{align}
  The two-sided estimates in \eqref{eq:easybounds2b} were proved, e.g., in Frank and Merz \cite[Theorem~10]{FrankMerz2023}.
  
The upper bound in \eqref{eq:heatkernelalpha1weightedsubordinatedboundsfinal} for $\alpha<2$ can be deduced from 
\cite[Corollary~3.8]{GrzywnyTrojan2021}. Private communication with the authors of that remarkable work indicates that their arguments should, however, be modified to prove the lower bound in \eqref{eq:heatkernelalpha1weightedsubordinatedboundsfinal}. So, for the sake of completeness, in Appendix~\ref{s:proofheatkernelalpha1subordinatedboundsfinal} below, we give a self-contained proof of the two-sided estimates in \eqref{eq:heatkernelalpha1weightedsubordinatedboundsfinal} for $\alpha\in (0,2)$.
\end{proof}

Using the pointwise bounds in Theorem~\ref{heatkernelalpha1subordinatedboundsfinal}, we show that $p_\zeta^{(\alpha)}$ is a strongly continuous contraction semigroup on $L^2(\R_+,r^{2\zeta}dr)$.

\begin{proposition}
  \label{strongcontinuity}
  Let $\zeta\in(-1/2,\infty)$ and $\alpha\in(0,2]$.
  Then, $\{p_\zeta^{(\alpha)}(t,\cdot,\cdot)\}_{t>0}$ is a strongly continuous contraction semigroup on $L^2(\R_+,r^{2\zeta}dr)$.
\end{proposition}

\begin{proof}
  By symmetry and the normalization \eqref{eq:normalizedalpha}, $p_\zeta^{(\alpha)}(t,\cdot,\cdot)$ defines a contraction on $L^2(\R_+,r^{2\zeta}dr)$ for every $t>0$ as a consequence of a Schur test\footnote{Recall that the Schur test for an integral operator $T$ with a non-negative integral kernel $K(x,y)$, $x\in X$, $y\in Y$ with $X,Y$ being two measurable spaces yields the operator norm bound $\|T\|_{L^2(X)\to L^2(Y)}\leq\sqrt{\alpha\beta}$, whenever $\sup_{x\in X}\int_Y dy\, K(x,y)\leq\alpha$ and $\sup_{y\in Y}\int_X dx\, K(x,y)\leq\beta$.}.
  
  To prove the strong continuity of $p_{\zeta}^{(\alpha)}(t,\cdot,\cdot)$, it suffices, by the density of $C_c^\infty(\R_+)$ in $L^2(\R_+,r^{2\zeta}dr)$, to show
  \begin{align}
    \label{eq:strongcontinuityaux}
    \lim_{t\searrow0} \int_0^\infty dr\, r^{2\zeta} \left|\int_0^\infty ds\, s^{2\zeta} p_\zeta^{(\alpha)}(t,r,s)\phi(s) - \phi(r)\right|^2 = 0
  \end{align}
  for every non-negative function $\phi\in C_c^\infty(\R_+)$.
  To that end, we use Jensen's inequality (in view of \eqref{eq:normalizedalpha}), i.e.,
  \begin{align*}
    \int_0^\infty dr\, r^{2\zeta}\left|\int_0^\infty ds\, s^{2\zeta} p_\zeta^{(\alpha)}(t,r,s)(\phi(s)-\phi(r))\right|^2
    \leq \int_0^\infty dr\, r^{2\zeta} \int_0^\infty ds\, s^{2\zeta}p_\zeta^{(\alpha)}(t,r,s)|\phi(s)-\phi(r)|^2
  \end{align*}
  and show that the right-hand side vanishes as $t\searrow0$.
  To that end, let $c:=c_\phi>0$ be such that $|\phi(r)|\lesssim_\phi\one_{r<c}$. Then, we consider four integrals over the regions
  $\{(r,s)\in\R_+\times\R_+:\, r\vee s\leq 2c_\phi\}$,
  $\{(r,s)\in\R_+\times\R_+:\, s\leq c_\phi\,, r\geq 2c_\phi\}$,
  $\{(r,s)\in\R_+\times\R_+:\, r\leq c_\phi,\, s\geq 2c_\phi\}$, and
  $\{(r,s)\in\R_+\times\R_+:\, r\wedge s\geq 2c_\phi\}$.
  Note that due to the compact support of $\phi$ and the symmetry of $p_\zeta^{(\alpha)}(t,r,s)$, it suffices to consider only the integrals over the first two sets. Suppose now $\alpha<2$. Then, by \eqref{eq:heatkernelalpha1weightedsubordinatedboundsfinal}, we obtain
  \begin{align*}
    \int_0^{2c} dr \int_0^{2c} ds\, (rs)^{2\zeta}p_\zeta^{(\alpha)}(t,r,s)|\phi(r)-\phi(s)|^2
    \lesssim_\phi t\int_0^{2c} dr \int_0^{2c} ds\, (rs)^{2\zeta} \frac{|r-s|^{1-\alpha}}{(r+s)^{2\zeta}}
    \lesssim_{\phi,\zeta,\alpha} t
  \end{align*}
  and
  \begin{align*}
    \int_{2c}^\infty dr \int_0^c ds\, (rs)^{2\zeta}p_\zeta^{(\alpha)}(t,r,s)|\phi(r)-\phi(s)|^2
    \lesssim_\phi t\int_{2c}^\infty dr \int_0^c ds\, \frac{(rs)^{2\zeta}}{|r-s|^{1+\alpha}(r+s)^{2\zeta}}
    \lesssim_{\phi,\zeta,\alpha}t.
  \end{align*}
  This shows \eqref{eq:strongcontinuityaux} for $\alpha<2$.
  Using \eqref{eq:easybounds2} instead of \eqref{eq:heatkernelalpha1weightedsubordinatedboundsfinal}, the case $\alpha=2$ is treated analogously.
  This concludes the proof.
\end{proof}

\section{3G inequality for $\alpha\in(0,2)$}
\label{s:3g}

The following 3G inequality for $p_\zeta^{(\alpha)}(t,r,s)$ is motivated and similar to that of Bogdan and Jakubowski \cite[(7)--(9)]{BogdanJakubowski2007}, which was important to study the heat kernels associated to fractional Laplacians with Hardy potentials in \cite{Bogdanetal2019,JakubowskiWang2020}.

\begin{theorem}
  \label{3gheatalpha}
  Let $\zeta\in[0,\infty)$, $\alpha\in(0,2)$. Then, for all $r,s,z,t,\tau>0$, we have
  \begin{align}
    \label{eq:3gheatalpha1}
    \begin{split}
      & \min\left\{ \! \left(\!\frac{r+z}{r+s+z}\!\right)^{2\zeta} p_\zeta^{(\alpha)}(t,r,z), \left(\!\frac{s+z}{r+s+z}\!\right)^{2\zeta} p_\zeta^{(\alpha)}(\tau,z,s) \! \right\}
        \!\lesssim_{\zeta,\alpha}\! p_\zeta^{(\alpha)}(t+\tau,r,s),
    \end{split}
  \end{align}
  and
  \begin{align}
    \label{eq:3gheatalpha2}
    & p_\zeta^{(\alpha)}(t,r,z)\cdot p_\zeta^{(\alpha)}(\tau,z,s) \\
    \notag
    & \ \lesssim_{\zeta,\alpha} p_\zeta^{(\alpha)}(t+\tau,r,s) \left[\! \left(\!\frac{r+s+z}{s+z}\!\right)^{2\zeta} p_\zeta^{(\alpha)}(t,r,z) + \left(\!\frac{r+s+z}{r+z}\!\right)^{2\zeta} p_\zeta^{(\alpha)}(\tau,z,s) \! \right]\!.
  \end{align}
\end{theorem}

\begin{remarks}
  (1) Weighted 3G inequalities for Green's functions (resolvent kernels) and their application to Schr\"odinger perturbations were studied, e.g., by Hansen \cite{Hansen2006}, with quasi-metric interpretations.
  \\
  (2) The weights in \eqref{eq:3gheatalpha1}--\eqref{eq:3gheatalpha2} are independent of $t$ and $\tau$, but they involve all three spatial variables $r,s,z$, which is slightly incompatible with the setting of \cite{Hansen2006}.
  \\
  (3) As observed in \cite[p.~182]{BogdanJakubowski2007}, there is no 3G inequality for the Brownian motion in Euclidean space $\R^d$. However, there is a substitute, called 4G inequality, in Bogdan and Szczypkowski \cite[Theorem~1.3]{BogdanSzczypkowski2014}.
\end{remarks}
  
\begin{proof}[Proof of Theorem~\ref{3gheatalpha}]
  Define and observe
  \begin{align}
    \label{eq:deffaux}
    \begin{split}
      f(r,s,z)
      & := \left(\frac{s+z}{r+s}\right)^{2\zeta}\one_{\{r>s\vee z\}} + \one_{\{r<s\vee z\}}
        \sim \left(\frac{s+z}{r+s+z}\right)^{2\zeta}.
    \end{split}
  \end{align}
  To prove \eqref{eq:3gheatalpha1}, without loss of generality, we assume $r>s$. Then,
  $f(r,s,z)= \left(\tfrac{s+z}{r+s}\right)^{2\zeta}\one_{\{r>z\}} + \one_{\{r<z\}}$
  and $f(s,r,z)=1$.
  By \eqref{eq:heatkernelalpha1weightedsubordinatedboundsfinal},
  \begin{align}
    p_\zeta^{(\alpha)}(t,r,z)
    \sim \min\left\{\frac{t}{|r-z|^{1+\alpha} (r+z)^{2\zeta}},\frac{1}{t^{1/\alpha}(t^{1/\alpha}+r+z)^{2\zeta}}\right\}.
  \end{align}
  Thus, the left-hand side of \eqref{eq:3gheatalpha1} is
  \begin{align}
    \scriptsize
    \label{eq:3gheatalphaaux1}
    \begin{split}
      & \left(f(s,r,z) \cdot p_\zeta^{(\alpha)}(t,r,z)\right) \wedge \left(f(r,s,z) \cdot p_\zeta^{(\alpha)}(\tau,z,s)\right) \\
      & \ \sim \min\left\{\! \frac{t}{|r-z|^{1+\alpha}(r+z)^{2\zeta}},\frac{\tau\cdot f(r,s,z)}{|s-z|^{1+\alpha}(s+z)^{2\zeta}}, \frac{1}{t^{\frac1\alpha}(t^{\frac1\alpha}+r+z)^{2\zeta}},\frac{f(r,s,z)}{\tau^{\frac1\alpha}(\tau^{\frac1\alpha}+s+z)^{2\zeta}}\right\}.
    \end{split}
  \end{align}
  We begin with estimating the minimum of the first two terms on the right-hand side of \eqref{eq:3gheatalphaaux1}. Since $(s+z)^{2\zeta}\one_{z>r>s}\geq r^{2\zeta}\one_{z>r>s}$, we get
  \begin{align}
    \label{eq:3gheatalphaaux2}
    \begin{split}
      & \min\left\{\frac{t}{|r-z|^{1+\alpha}(r+z)^{2\zeta}},\frac{\tau f(r,s,z)}{|s-z|^{1+\alpha}(s+z)^{2\zeta}}\right\} \\
      & \quad \lesssim \frac{t+\tau}{r^{2\zeta}}\min\left\{\frac{1}{|r-z|^{1+\alpha}},\frac{1}{|s-z|^{1+\alpha}}\right\} \\
      & \quad \lesssim \frac{t+\tau}{r^{2\zeta}\cdot|r-s|^{1+\alpha}}
      \sim \frac{t+\tau}{|r-s|^{1+\alpha}(r+s)^{2\zeta}},
    \end{split}
  \end{align}
  which is half of our desired estimate. We now consider the minimum of the last two terms in \eqref{eq:3gheatalphaaux1} and claim
  \begin{align}
    \label{eq:3gheatalphaaux3}
    \begin{split}
      & \min\left\{\frac{1}{t^{1/\alpha}(t^{1/\alpha}+r)^{2\zeta}},\frac{f(r,s,z)}{\tau^{1/\alpha}(\tau^{1/\alpha}+s+z)^{2\zeta}}\right\} \\
      & \quad \lesssim \frac{1}{t^{1/\alpha}(t^{1/\alpha}+r)^{2\zeta} + \tau^{1/\alpha}(\tau^{1/\alpha}+r)^{2\zeta}}.
    \end{split}
  \end{align}
  Suppose \eqref{eq:3gheatalphaaux3} is true. Then, by
  \begin{align}
    \begin{split}
      (t+\tau)^{1/\alpha}((t+\tau)^{1/\alpha}+r+s)^{2\zeta}
      & \sim (t+\tau)^{1/\alpha}((t+\tau)^{1/\alpha}+r)^{2\zeta} \\
      & \lesssim (t+\tau)^{(1+2\zeta)/\alpha} + (t+\tau)^{1/\alpha}r^{2\zeta} \\
      & \lesssim t^{1/\alpha}(t^{2\zeta/\alpha}+r^{2\zeta}) + \tau^{1/\alpha}(\tau^{2\zeta/\alpha}+r^{2\zeta}) \\
      & \sim t^{1/\alpha}(t^{1/\alpha}+r)^{2\zeta} + \tau^{1/\alpha}(\tau^{1/\alpha}+r)^{2\zeta},
    \end{split}
  \end{align}
  we obtain
  \begin{align}
    \label{eq:3gheatalphaaux4}
    \begin{split}
      & \min\left\{\frac{1}{t^{\frac1\alpha}(t^{\frac1\alpha}+r+z)^{2\zeta}},\frac{f(r,s,z)}{\tau^{\frac1\alpha}(\tau^{\frac1\alpha}+s+z)^{2\zeta}}\right\}
      \lesssim \frac{1}{(t+\tau)^{\frac1\alpha}((t+\tau)^{\frac1\alpha}+r+s)^{2\zeta}},
    \end{split}
  \end{align}
  which, together with \eqref{eq:3gheatalphaaux2}, completes the proof of \eqref{eq:3gheatalpha1}. It remains to prove \eqref{eq:3gheatalphaaux3}. We distinguish between $z<r$ and $z>r$. In the latter case, we have
  \begin{align}
    \label{eq:3gheatalphaaux5}
    \begin{split}
      & \min\left\{\frac{1}{t^{\frac1\alpha}(t^{\frac1\alpha}+r+z)^{2\zeta}},\frac{1}{\tau^{\frac1\alpha}(\tau^{\frac1\alpha}+s+z)^{2\zeta}}\right\} \\
      & \quad \leq \min\left\{\frac{1}{t^{\frac1\alpha}(t^{\frac1\alpha}+r)^{2\zeta}},\frac{1}{\tau^{\frac1\alpha}(\tau^{\frac1\alpha}+r)^{2\zeta}}\right\},
    \end{split}
  \end{align}
  as desired. On the other hand, if $z<r$, then we have, using $s\vee z<r$,
  \begin{align*}
    \frac{(s+z)(\tau^{1/\alpha}+r)}{(r+s)(\tau^{1/\alpha}+s+z)}
    \sim \frac{(s+z)\tau^{1/\alpha} + r(s+z)}{r\tau^{1/\alpha} + r(s+z)}
    \leq 3,
  \end{align*}
  and so
  \begin{align}
    \label{eq:3gheatalphaaux6}
    \begin{split}
      & \min\left\{\frac{1}{t^{1/\alpha}(t^{1/\alpha}+r+z)^{2\zeta}},\frac{[(s+z)/(r+s)]^{2\zeta}}{\tau^{1/\alpha}(\tau^{1/\alpha}+s+z)^{2\zeta}}\right\} \\
      & \quad \lesssim \min\left\{\frac{1}{t^{1/\alpha}(t^{1/\alpha}+r)^{2\zeta}},\frac{1}{\tau^{1/\alpha}(\tau^{1/\alpha}+r)^{2\zeta}}\right\} \\
      & \quad \sim \frac{1}{t^{1/\alpha}(t^{1/\alpha}+r)^{2\zeta} + \tau^{1/\alpha}(\tau^{1/\alpha}+r)^{2\zeta}},
    \end{split}
  \end{align}
  which shows \eqref{eq:3gheatalphaaux3} for $z<r$. Plugging \eqref{eq:3gheatalphaaux2} and \eqref{eq:3gheatalphaaux4} into \eqref{eq:3gheatalphaaux1} yields
  \begin{align}
    \begin{split}
      & \left(f(s,r,z)\cdot p_\zeta^{(\alpha)}(t,r,z)\right) \wedge \left(f(r,s,z) \cdot p_\zeta^{(\alpha)}(\tau,z,s)\right) \\
      & \quad \lesssim \min\left\{\frac{t+\tau}{|r-s|^{1+\alpha}(r+s)^{2\zeta}}, \frac{1}{(t+\tau)^{1/\alpha}((t+\tau)^{1/\alpha}+r+s)^{2\zeta}}\right\} \\
      & \quad \sim p_\zeta^{(\alpha)}(t+\tau,r,s),
    \end{split}
  \end{align}
  as claimed. Estimate \eqref{eq:3gheatalpha2} follows from \eqref{eq:3gheatalpha1}.
\end{proof}

In our work \cite{Bogdanetal2024} with Tomasz Jakubowski, we prove and use the following variant of the above 3G inequality, which is important to prove the continuity of Schr\"odinger perturbations of $p_\zeta^{(\alpha)}(t,r,s)$ by Hardy potentials. In particular, it holds for all $\zeta>-1/2$.

\begin{theorem}[{\cite[Lemma~2.5]{Bogdanetal2024}}]
  \label{lem:3Gineq}
  Let $\zeta\in(-1/2,\infty)$ and $\alpha\in(0,2)$. Then, for all $t,\tau,t,s,r,z>0$, we have
  \begin{align}
    \frac{p_\zeta^{(\alpha)}(t,r,z) p_\zeta^{(\alpha)}(\tau,z,s)}{p^{(\alpha)}(t+\tau,r,s)}
    \lesssim_{\zeta,\alpha} \frac{p_\zeta^{(\alpha)}(t,r,z)}{(\tau^{1/\alpha}+z+s)^{2\zeta}} + \frac{p_\zeta^{(\alpha)}(\tau,z,s)}{(t^{1/\alpha}+z+r)^{2\zeta}}.
  \end{align}
\end{theorem}

For the sake of completeness, we give the short proof using the notation
\begin{align}
  \label{eq:stabledensity}
  p^{(\alpha)}(t,r,s) := \frac{1}{2\pi}\int_{-\infty}^\infty \me{-i(s-r)z} \me{-t|z|^\alpha}\,dz
\end{align}
for the stable density on the real line.

\begin{proof}
  By Theorem~\ref{heatkernelalpha1subordinatedboundsfinal},
  \begin{align*}
    p_\zeta^{(\alpha)}(t,r,s) \sim_{\zeta,\alpha} p^{(\alpha)}(t,r,s) (t^{1/\alpha} + r+s)^{-2\zeta}.
  \end{align*}
  Thus, the claim follows from the 3G inequality for $p^{(\alpha)}(t,r,s)$ in \cite[(7)--(9)]{BogdanJakubowski2007}.
\end{proof}

\section{Comparability results for $p_\zeta^{(\alpha)}$}
\label{s:comparison}

The following comparability results are crucial for our work \cite{Bogdanetal2024}. As we will see, they follow from the bounds in Theorem~\ref{heatkernelalpha1subordinatedboundsfinal}.

\begin{theorem}
  \label{comparablealpha}
  Let $\zeta\in(-1/2,\infty)$ and $\alpha\in(0,2]$.
  \begin{enumerate}
  \item Let $z,s>0$, $0<C\leq 1$, and $\tau\in[C,C^{-1}]$. Then, there is $c_j=c_j(\zeta,C)$, $j\in\{1,2\}$ with
    \begin{align}
      \label{eq:comparablealpha1}
      \begin{split}
        p_\zeta^{(2)}(1,c_1 z,c_1 s) & \lesssim_{C,\zeta} p_\zeta^{(2)}(\tau,z,s) \lesssim_{C,\zeta} p_\zeta^{(2)}(1,c_2 z,c_2 s), \\
        p_\zeta^{(\alpha)}(\tau,z,s) & \sim_{C,\zeta,\alpha} p_\zeta^{(\alpha)}(1,z,s), \quad \alpha<2.
      \end{split}
    \end{align}
    In particular, for $\alpha\in(0,2)$ and $\tau,z,s,c>0$, one has
    \begin{align}
      \label{eq:comparablealpha7}
      \begin{split}
        p_\zeta^{(\alpha)}(\tau,cz,cs) \sim_{\zeta,\alpha,c} p_\zeta^{(\alpha)}(\tau,z,s), \\
        p_\zeta^{(\alpha)}(c\tau,z,s) \sim_{\zeta,\alpha,c} p_\zeta^{(\alpha)}(\tau,z,s), \quad \alpha<2.
      \end{split}
    \end{align}
    
  \item Let $C,\tau>0$ and $0<z\leq s/2<\infty$. Then, there is $c=c(\zeta,C)$ with
    \begin{subequations}
      \label{eq:comparablealpha2}
      \begin{align}
        \label{eq:comparablealpha2a}
        \begin{split}
          p_\zeta^{(\alpha)}(\tau,z,s)
          & \lesssim_{C,\zeta,\alpha} p_\zeta^{(\alpha)}(\tau,c,c s)\one_{\{\tau>C\}} \\
          & \quad + \left(\tau^{-\frac12}\frac{\me{-cs^2/\tau}}{(\tau+s^2)^{\zeta}}\one_{\alpha=2} + \frac{\tau}{s^{2\zeta+1+\alpha}+\tau^{\frac{2\zeta+1+\alpha}{\alpha}}} \one_{\alpha\in(0,2)}\right)\one_{\{\tau<C\}},
        \end{split}
        \\
        \label{eq:comparablealpha2b}
        p_\zeta^{(\alpha)}(\tau,z,s)
        & \lesssim_{\zeta,\alpha} \tau^{-\frac12}\frac{\me{-cs^2/\tau}}{(\tau+s^2)^{\zeta}}\one_{\alpha=2} + \frac{\tau}{s^{2\zeta+1+\alpha} + \tau^{(2\zeta+1+\alpha)/\alpha}}\one_{\alpha\in(0,2)}, \\
        \label{eq:comparablealpha2c}
        p_\zeta^{(\alpha)}(\tau,z,s)
        & \lesssim_{\zeta,\alpha} s^{-(2\zeta+1)}.
      \end{align}
    \end{subequations}

  \item Let $0<\tau\leq1$, $0<z\leq s/2$, and $s\geq C>0$. Then, there is $c=c(\zeta,C)$ with
    \begin{align}
      \label{eq:comparablealpha3}
      p_\zeta^{(\alpha)}(\tau,z,s) \lesssim_{C,\zeta,\alpha} p_\zeta^{(\alpha)}(1,c,c s).
    \end{align}

  \item Let $0<r\leq s$. Then, there is $c=c(\zeta)$ with
    \begin{align}
      \label{eq:comparablealpha5}
      p_\zeta^{(\alpha)}(1,1,s) \lesssim_{\zeta,\alpha} p_\zeta^{(\alpha)}(1,c r,c s).
    \end{align}
    In particular, for all $r,s>0$,
    \begin{align}
      \label{eq:comparablealpha4}
      \min\{p_\zeta^{(\alpha)}(1,1,r),p_\zeta^{(\alpha)}(1,1,s)\}
      \lesssim_{\zeta,\alpha} p_\zeta^{(\alpha)}(1,c r,c s).
    \end{align}

  \item Let $r,s,z,t>0$ with $|z-s|>|r-s|/2$. Then, there is $c=c(\zeta)$ with
    \begin{align}
      \label{eq:comparablealpha6}
      p_\zeta^{(\alpha)}(t,z,s) \lesssim_{\zeta,\alpha} p_\zeta^{(\alpha)}(t,cr,cs).
    \end{align}
  \end{enumerate}
\end{theorem}

The constants $c$ in the arguments of functions in Theorem~\ref{comparablealpha} may change from place to place. Note that in the above estimates the point $z=1$ is a natural reference point for the heat kernel $p_\zeta^{(\alpha)}(\tau,z,s)$ when $s\gg z$ and $\tau\sim1$ by the spatial homogeneity, i.e., the scaling \eqref{eq:scalingalpha}, of $p_\zeta^{(\alpha)}$ (see, e.g., \eqref{eq:comparablealpha3}, where $s\gg z$ and $z$ is replaced by a constant of order $1$). This is in contrast to the analysis of the heat kernel $\me{-\tau(-\Delta)^{\alpha/2}}(z,y)$ in $\R^d$, where $z=0$ is a natural reference point when $\tau\sim1$ and $|y|\gg|z|$ because in this setting we have $\me{-\tau(-\Delta)^{\alpha/2}}(z,y) \lesssim \me{-\tau(-\Delta)^{\alpha/2}}(0,y)$ in view of the translation invariance, i.e., the fact that $\me{-\tau(-\Delta)^{\alpha/2}}(z,y)$ only depends on $z$ and $y$ via the difference $|y-z|$.
We give some interpretations of the bounds in Theorem~\ref{comparablealpha} after the proof.

\begin{proof}
  \begin{enumerate}
  \item The estimates follow from \eqref{eq:easybounds2} for $\alpha=2$ and from \eqref{eq:heatkernelalpha1weightedsubordinatedboundsfinal} if $\alpha<2$.
    
    \smallskip
  \item We start with $\alpha=2$. We first consider \eqref{eq:comparablealpha2a}. By \eqref{eq:easybounds2a},
    \begin{align}
      p_\zeta^{(2)}(\tau,z,s)
      \asymp \tau^{-\frac12}\frac{\exp(-c(s-z)^2/\tau)}{(\tau+sz)^\zeta}
      \asymp \tau^{-\frac12}\frac{\exp(-cs^2/\tau)}{(\tau+sz)^\zeta}.
    \end{align}
    If $\tau>C>0$, then, by $2s/\tau\leq(1+s^2)/\tau\lesssim_C 1+s^2/\tau$,
    \begin{align}
      \begin{split}
        \frac{(\tau+s)^\zeta}{(\tau+sz)^\zeta}\me{-cs^2/\tau}
        & \leq \left(1+\frac s\tau\right)^\zeta \me{-cs^2/\tau}
          \lesssim_C \left(1+\frac{s^2}{\tau}\right)^\zeta \me{-cs^2/\tau}
          \lesssim_\zeta \me{-cs^2/\tau} \\
        & \sim_C \me{-c(s-1)^2/\tau},
      \end{split}
    \end{align}
    which proves \eqref{eq:comparablealpha2a} for $\tau>C>0$.
    Similarly, we have for all $\tau>0$,
    \begin{align}
      \frac{(\tau+s^2)^\zeta}{(\tau+sz)^\zeta}\me{-cs^2/\tau}
      \leq (1+s^2/\tau)^\zeta \me{-cs^2/\tau}
      \lesssim \me{-cs^2/\tau},
    \end{align}
    which yields \eqref{eq:comparablealpha2b} and the second part of \eqref{eq:comparablealpha2a}, where $\tau<C$.
    Estimate \eqref{eq:comparablealpha2c} follows from
    \begin{align}
      p_\zeta^{(2)}(\tau,z,s)
      \lesssim \left(\frac{s^2}{\tau}\right)^{\zeta+\frac12} \exp(-cs^2/\tau) \cdot s^{-(2\zeta+1)}
      \lesssim s^{-(2\zeta+1)}.
    \end{align}
    This concludes the proof of \eqref{eq:comparablealpha2a}--\eqref{eq:comparablealpha2c} for $\alpha=2$.
    
    Now we verify \eqref{eq:comparablealpha2} for $\alpha<2$. We start with \eqref{eq:comparablealpha2a}. By \eqref{eq:heatkernelalpha1weightedsubordinatedboundsfinal},
    \begin{align}
      p_\zeta^{(\alpha)}(\tau,z,s)
      \sim \frac{\tau}{s^{2\zeta+1+\alpha} + \tau^{1+1/\alpha}(\tau^{1/\alpha}+s)^{2\zeta}}.
    \end{align}
    Thus, \eqref{eq:comparablealpha2a} for $\tau>C>0$ follows from
    \begin{align}
      \label{eq:comparablealpha2aaux1}
      \begin{split}
        & \frac{|s-1|^{1+\alpha}(s+1)^{2\zeta}}{s^{2\zeta+1+\alpha} + \tau^{1+1/\alpha}(\tau^{1/\alpha}+s)^{2\zeta}}
        + \frac{\tau^{1+1/\alpha}(\tau^{1/\alpha}+s+1)^{2\zeta}}{s^{2\zeta+1+\alpha} + \tau^{1+1/\alpha}(\tau^{1/\alpha}+s)^{2\zeta}} \\
        & \quad \lesssim \frac{s^{1+\alpha}(1+s)^{2\zeta} + (s+1)^{2\zeta}}{s^{2\zeta+1+\alpha}+1} + 1 + \frac{1}{(\tau^{1/\alpha}+s)^{2\zeta}} \\
        & \quad \lesssim_{C,\zeta} 1.
      \end{split}
    \end{align}
    Similarly, for all $\tau>0$,
    \begin{align}
      p_\zeta^{(\alpha)}(\tau,z,s)
      \lesssim \frac{\tau}{s^{2\zeta+1+\alpha} + \tau^{\frac{2\zeta+1+\alpha}{\alpha}}},
    \end{align}
    which yields \eqref{eq:comparablealpha2b} and the second part of \eqref{eq:comparablealpha2a}. Estimate \eqref{eq:comparablealpha2c} follows from
    \begin{align}
      \begin{split}
        p_\zeta^{(\alpha)}(\tau,z,s)
        & \lesssim \frac{\tau}{s^{2\zeta+1+\alpha} + \tau^{(1+\alpha)/\alpha}\cdot s^{2\zeta}} \\
        & \lesssim \frac{\tau}{s^{2\zeta+1+\alpha}}\one_{\{\tau<s^\alpha\}} + \frac{\tau}{\tau^{(1+\alpha)/\alpha}\cdot s^{2\zeta}}\one_{\{\tau>s^\alpha\}}
        \lesssim s^{-(2\zeta+1)}.
      \end{split}
    \end{align}
    This concludes the proof of \eqref{eq:comparablealpha2a}--\eqref{eq:comparablealpha2c} for $\alpha<2$.

    \smallskip
  \item To prove \eqref{eq:comparablealpha3}, we begin with $\alpha=2$. By \eqref{eq:easybounds2a}, we have, for $\tau\in(0,1]$, $0<z\leq s/2$, and $s\geq C>0$,
    \begin{align}
      \begin{split}
        p_\zeta^{(2)}(\tau,z,s)
        & \asymp \tau^{-\frac12}\frac{\exp\left(-\frac{c(z-s)^2}{\tau}\right)}{(zs+\tau)^{\zeta}}
          \lesssim \frac{\exp(-cs^2/\tau)}{\tau^{\zeta+1/2}} \cdot \frac{s^{2\zeta+1}}{s^{2\zeta+1}} \\
        & \lesssim_C \frac{(s^2/\tau)^{\zeta+1/2} \exp(-cs^2/\tau)}{s^{\zeta}}
          \lesssim_C \frac{\me{-cs^2/\tau}}{(s+1)^{\zeta}}
          \lesssim \frac{\me{-c(s-1)^2}}{(s+1)^{\zeta}}.
      \end{split}
    \end{align}
    This concludes the proof of \eqref{eq:comparablealpha3} for $\alpha=2$.
    
    If $\alpha<2$, we have, for $\tau\in(0,1]$, $0<z\leq s/2$, and $s\geq C>0$,
    \begin{align}
      \begin{split}
        p_\zeta^{(\alpha)}(\tau,z,s)
        & \sim_{\zeta} \frac{\tau}{s^{2\zeta+1+\alpha} + \tau^{1+1/\alpha}(\tau^{1/\alpha}+s)^{2\zeta}}.
      \end{split}
    \end{align}
    Thus, \eqref{eq:comparablealpha3} follows from
    \begin{align}
      \label{eq:comparablealpha3aux1}
      \begin{split}
        & \frac{\tau\left(|s-1|^{1+\alpha}(s+1)^{2\zeta} + (1+s)^{2\zeta}\right)}{s^{2\zeta+1+\alpha} + \tau^{1+1/\alpha}(\tau^{1/\alpha}+s)^{2\zeta}}
          \lesssim \frac{(1+s)^{2\zeta}(1 + s^{1+\alpha})}{s^{2\zeta+1+\alpha}}
          \lesssim_C 1.
      \end{split}
    \end{align}
    This concludes the proof of \eqref{eq:comparablealpha3} for $\alpha<2$.

    \smallskip
  \item We come to \eqref{eq:comparablealpha5}, with 
    $0<r\leq s$.
    We first treat $\alpha=2$. By \eqref{eq:easybounds2a},
    \begin{align}
      \begin{split}
        p_\zeta^{(2)}(1,1,s)
        & \asymp_{\zeta} \frac{\me{-c(s-1)^2}}{(s+1)^\zeta}
          \lesssim \frac{(1+rs)^\zeta \me{-cs^2}}{(s+1)^\zeta (1+rs)^\zeta}
          \leq \frac{(1+s^2)^\zeta \me{-cs^2}}{(1+rs)^\zeta}
          \lesssim_\zeta \frac{\me{-cs^2}}{(1+rs)^\zeta} \\
        & \lesssim \frac{\me{-c(r-s)^2}}{(rs+1)^\zeta}
          \asymp_{\zeta} p_\zeta^{(2)}(1,cr,cs),
      \end{split}
    \end{align}
    where we used $2s^2\geq (r-s)^2$. This concludes the proof of \eqref{eq:comparablealpha5} for $\alpha=2$.
    
    If $\alpha<2$, then, by \eqref{eq:heatkernelalpha1weightedsubordinatedboundsfinal},
    \begin{align}
      \begin{split}
        p_\zeta^{(\alpha)}(1,1,s)
        & \sim_{\zeta} \frac{1}{|s-1|^{1+\alpha}(1+s)^{2\zeta} + (1+s)^{2\zeta}}.
      \end{split}
    \end{align}
    Since $2s^2\geq (r-s)^2$, we have
    \begin{align}
      \begin{split}
        & \frac{|r-s|^{1+\alpha}(r+s)^{2\zeta} + (1+r+s)^{2\zeta}}{|s-1|^{1+\alpha}(1+s)^{2\zeta} + (1+s)^{2\zeta}} \\
        & \quad \lesssim \frac{s^{2\zeta+1+\alpha} + (1+s)^{2\zeta}}{|s-1|^{1+\alpha}(1+s)^{2\zeta} + (1+s)^{2\zeta}}
        \lesssim 1,
      \end{split}
    \end{align}
    which concludes the proof of \eqref{eq:comparablealpha5} for $\alpha<2$.
    To prove \eqref{eq:comparablealpha4}, it suffices, by
  %
    symmetry, to assume $s\geq r>0$. Then, the claim follows from \eqref{eq:comparablealpha5} since $\min\{p_\zeta^{(\alpha)}(1,1,r),p_\zeta^{(\alpha)}(1,1,s)\} \leq p_\zeta^{(\alpha)}(1,1,s)$.
    
    
    
    \smallskip
  \item We come to \eqref{eq:comparablealpha6}.
    If $z>r$, the estimate follows immediately from \eqref{eq:easybounds2a} if $\alpha=2$ and from \eqref{eq:heatkernelalpha1weightedsubordinatedboundsfinal} if $\alpha<2$. So, suppose $z<r$. We distinguish now between $z>s$ and $z<s$ and start with the former case. Then, the assumption $|z-s|>|r-s|/2$ implies
    \begin{align}
      z = z-s+s\geq s+\frac{|r-s|}{2}
      = \frac{r+s}{2}
      \gtrsim r.
    \end{align}
    Thus, $z\gtrsim r$ and we can again use \eqref{eq:easybounds2a} for $\alpha=2$ and \eqref{eq:heatkernelalpha1weightedsubordinatedboundsfinal} for $\alpha<2$ to conclude the estimate in this case.
    Thus, we are left to treat $z<s$.
    We first let $\alpha=2$ and distinguish between the three subcases $s>2r$, $s<r/2$, and $s\in(r/2,2r)$. To treat $s>2r$ and $s<r/2$, we estimate
    \begin{align}
      \begin{split}
        \frac{\exp\left(-\frac{c(z-s)^2}{t}\right)}{(zs+t)^\zeta}
        & \leq \frac{(rs/t+1)^\zeta}{(zs/t+1)^\zeta} \cdot \frac{\exp\left(-\frac{c(r-s)^2}{4t}\right)}{(rs+t)^\zeta} \\
        & \leq \left(1+\frac{rs}{t}\right)^\zeta \cdot \exp\left(-\frac{c(r-s)^2}{8t}\right) \cdot \frac{\exp\left(-\frac{c(r-s)^2}{8t}\right)}{(rs+t)^\zeta} \\
        & \lesssim_\zeta \frac{\exp\left(-\frac{c(r-s)^2}{8t}\right)}{(rs+t)^\zeta},
      \end{split}
    \end{align}
    as desired. It remains to treat the case $s\in(r/2,2r)$ and $z<r\vee s$. We distinguish between $s>2z$ and $s\in(z,2z)$. If $s>2z$, then we can argue similarly as before and obtain
    \begin{align}
      \begin{split}
        \frac{\exp\left(-\frac{c(z-s)^2}{t}\right)}{(zs+t)^\zeta}
        & \leq \left(1+\frac{rs}{t}\right)^\zeta \cdot \exp\left(-\frac{c(z-s)^2}{8t}\right) \cdot \frac{\exp\left(-\frac{c(r-s)^2}{8t}\right)}{(rs+t)^\zeta} \\
        & \leq \left(1+\frac{2s^2}{t}\right)^\zeta \cdot \exp\left(-\frac{cs^2}{8t}\right) \cdot \frac{\exp\left(-\frac{c(r-s)^2}{8t}\right)}{(rs+t)^\zeta} \\
        & \lesssim_\zeta \frac{\exp\left(-\frac{c(r-s)^2}{8t}\right)}{(rs+t)^\zeta}.
      \end{split}
    \end{align}
    Finally, for $s\in(r/2,2r)$, $z<r$, and $s\in(z,2z)$, we use $zs>s^2/2>rs/4$ to get the desired estimate in the denominator in \eqref{eq:easybounds2a}. This completes the proof of \eqref{eq:comparablealpha6} for $\alpha=2$.
    
    Now we deal with $\alpha<2$ and $z<r\wedge s$.
    Observe that in \eqref{eq:heatkernelalpha1weightedsubordinatedboundsfinal} it suffices to estimate $z+s\gtrsim s+r$ and $z^2+s^2\gtrsim r^2+s^2$. These estimates follow immediately if $s>r$. So, suppose $z<s<r$. Then, by the assumption $|z-s|=s-z>|r-s|/2=(r-s)/2$, we get $3s/2>z+r/2>r/2$. This completes the proof of \eqref{eq:comparablealpha6}.
  \end{enumerate}
  The proof of Theorem~\ref{comparablealpha} is concluded.
\end{proof}

We close with some interpretations of the bounds in Theorem~\ref{comparablealpha}. The numbering of the following remarks refers to the numbering of the bounds in Theorem~\ref{comparablealpha}.
\begin{remarks}
  (1) For fixed time scales $\tau\sim1$, the heat kernels $p_\zeta^{(\alpha)}(\tau,z,s)$ and $p_\zeta^{(\alpha)}(1,r,s)$ are comparable for all $z,s>0$.
  \\
  (2) Spatial or temporal dilations on the unit order are negligible when $\alpha<2$.
  \\
  (3) Suppose $s\geq2z$. For large times $\tau>C$, the heat kernel $p_\zeta^{(\alpha)}(\tau,z,s)$ can be replaced with $p_\zeta^{(\alpha)}(\tau,1,s)$, even if $z$ is small, thanks to the lower-boundedness of $\tau$ and the large distance between $z$ and $s$. For small times $\tau<C$, the distance $|s-z|$ can still be replaced with $s$. However, for small $z$ and $\tau$, we cannot compare $p_\zeta^{(\alpha)}(\tau,z,s)$ with $p_\zeta^{(\alpha)}(\tau,1,s)$. At least when $\alpha<2$, we can $p_\zeta^{(\alpha)}(\tau,z,s)$ compare with $p_\zeta^{(\alpha)}(\tau,0,s)$.
  \\
  (4) For small times $\tau\in(0,1]$ and locations $s\geq2z$, which are strictly away from zero, the heat kernel $p_\zeta^{(\alpha)}(\tau,z,s)$ can be estimated from above by $p_\zeta^{(\alpha)}(1,1,s)$. This is because the lower-boundedness of $s$ prevents the heat kernel from blowing up for small times and small $z$, and the large separation between $s$ and $z$ allows to replace $|z-s|$ and $z+s$ by $s$.
  \\
  (5) For fixed time $\tau=1$ and location $s>0$, it is more likely to go from $s$ to $r\leq s$ than from $s$ to $1$. This is plausible for $r>1$, while, if $r\leq1$, it does not really matter if we go to $1$ or to $r$ within a unit time step.
  \\
  (6) The probability for going from $s$ to $z$ is bounded by that for going from $s$ to $r$ when $|z-s|$ is greater than $|r-s|$. Moreover, the (killing or mass creating) effect from the boundary, i.e., the origin, is negligible.
\end{remarks}

\appendix

\section{Proof of Theorem~\ref{heatkernelalpha1subordinatedboundsfinal} for $\alpha\in(0,2)$}
\label{s:proofheatkernelalpha1subordinatedboundsfinal}
Recall our claim \eqref{eq:heatkernelalpha1weightedsubordinatedboundsfinal}. Namely, if
$\alpha\in(0,2)$ and $\zeta>-1/2$, then
\begin{align*}
  p_\zeta^{(\alpha)}(t,r,s)
  & \sim_{\zeta,\alpha} \frac{t}{|r-s|^{1+\alpha}(r+s)^{2\zeta} + t^{\frac{1+\alpha}{\alpha}}(t^{\frac1\alpha}+r+s)^{2\zeta}}\quad \mbox{ for $r,s,t>0$.}
\end{align*}
By the scaling \eqref{eq:scalingalpha},  
$t=1$ suffices. We split the proof into several steps.

\subsection{Auxiliary bounds}\label{A.1}

We first prove the following, rather bulky, bounds, namely
\begin{align}
  \begin{split}
    \label{eq:heatkernelalpha1subordinatedboundsfinalaux1}
    p_\zeta^{(\alpha)}(1,r,s)
    & \sim_{\zeta,\alpha} \one_{rs<1}\one_{(r-s)^2<1} \\
    & \qquad + (rs)^{-\zeta}\one_{(r-s)^2<1<rs} \\
    & \qquad + |r-s|^{-(2\zeta+1+\alpha)}\one_{rs<1<(r-s)^2} \\
    & \qquad + \one_{rs>1}\one_{(r-s)^2>1} \left(\! \frac{\one_{rs<(r-s)^2}}{|r-s|^{2\zeta+1+\alpha}} + \frac{\one_{rs>(r-s)^2}}{(rs)^{\zeta} |r-s|^{1+\alpha}} \!\right).
    \end{split}
\end{align}
To prove \eqref{eq:heatkernelalpha1subordinatedboundsfinalaux1}, we use the subordinator bounds
\begin{align}
  \sigma_{1}^{(\alpha/2)}(\tau)
  \sim_\alpha \frac{\exp\left(-C(\alpha)\tau^{-c_1}\right)}{\tau^{c_2}}\one_{\tau<1} + \tau^{-1-\alpha/2} \one_{\tau>1},
\end{align}
with $C(\alpha)>0$, $c_1=c_1(\alpha)=\alpha/(2-\alpha)\in(0,\infty)$, and $c_2=c_2(\alpha)=(2-\alpha/2)/(2-\alpha)\in(1,\infty)$; see, e.g., \cite[Proposition~B.1]{BogdanMerz2024}.
Since $\R_+\ni x\mapsto\me{-1/x^s}$ vanishes at zero faster than any polynomial whenever $s>0$, there are $C_1(\alpha),C_2(\alpha)>0$ with $C_1(\alpha)>C_2(\alpha)$ such that
\begin{align}
  \label{eq:subordinatorboundsconsequence}
  \frac{\exp\left(-C_1(\alpha)\tau^{-c_1}\right)}{\tau^{1+\alpha/2}}
  \lesssim_\alpha \sigma_{1}^{(\alpha/2)}(\tau)
  \lesssim_\alpha \frac{\exp\left(-C_2(\alpha)\tau^{-c_1}\right)}{\tau^{1+\alpha/2}}.
\end{align}
From the following computations, it will transpire that the precise rate of the exponential decay of $\sigma_1^{(\alpha/2)}(\tau)$ at $\tau=0$ is irrelevant. This motivates us to not distinguish between $C_1(\alpha)$ and $C_2(\alpha)$ in \eqref{eq:subordinatorboundsconsequence} in the following and simply write
\begin{align}
  \label{eq:subordinatorboundsconsequence2}
  \sigma_{1}^{(\alpha/2)}(\tau)
  \asymp_\alpha \frac{\exp\left(-C(\alpha)\tau^{-c_1}\right)}{\tau^{1+\alpha/2}}.
\end{align}
We recall that the notation $\asymp$ means the same as $\sim$, but constants in the argument of the exponential function (like $C(\alpha)$ in \eqref{eq:subordinatorboundsconsequence2}) may be different in the upper and lower bounds.
For simplicity, in the following, we will not explicitly indicate any parameter dependence of constants appearing as prefactors or in exponentials anymore, so we will write, e.g., $c$ for $c_\zeta$.
Then, Formulae \eqref{eq:subordination}, \eqref{eq:easybounds2}, and \eqref{eq:subordinatorboundsconsequence2} yield, for some $c,c_1>0$,
\begin{align}
  \label{eq:heatkernelalphasubordinatedbounds}
  \begin{split}
    p_\zeta^{(\alpha)}(1,r,s)
    & \asymp \int_0^\infty \frac{d\tau}{\tau} \tau^{-\frac{1+\alpha}{2}} \left[\frac{1}{\tau^{\zeta}}\one_{rs\leq\tau} + \frac{\one_{rs\geq\tau}}{(rs)^{\zeta}}\right] \exp\left(\!-c\left(\frac{(r-s)^2}{\tau}+\frac{1}{\tau^{c_1}}\right)\!\right).
  \end{split}
\end{align}
We now estimate the integral on the right of \eqref{eq:heatkernelalphasubordinatedbounds}. To this end, we distinguish between four cases based on $rs\lessgtr1$ and $|r-s|\lessgtr1$. Although the notation will not emphasize it, we repeat once more that the following estimates are not uniform in $\zeta$ or $\alpha$.

\subsubsection{Case $rs\vee(r-s)^2<1$}

We show that $p_\zeta^{(\alpha)}(1,r,s)\sim1$. We consider the first summand in \eqref{eq:heatkernelalphasubordinatedbounds} and estimate
\begin{align}
  \label{eq:heatkernelalphasubordinatedboundsaux1}
  \begin{split}
    & \int_{rs}^\infty \frac{d\tau}{\tau} \tau^{-\zeta-\frac{1+\alpha}{2}} \me{-c((r-s)^2/\tau + 1/\tau^{c_1})} \\
    & \quad \asymp \left[\int_{rs}^1 \frac{d\tau}{\tau} \tau^{-\frac{2\zeta+1+\alpha}{2}} \me{-c((r-s)^2/\tau+\tau^{-c_1})} + \int_{1}^\infty \frac{d\tau}{\tau} \tau^{-\frac{2\zeta+1+\alpha}{2}}\right].
  \end{split}
\end{align}
The second summand in the last line of \eqref{eq:heatkernelalphasubordinatedboundsaux1} is $\sim1$ as desired. Since the first summand on the right-hand side of~\eqref{eq:heatkernelalphasubordinatedboundsaux1} is non-negative, we get that the left-hand side of~\eqref{eq:heatkernelalphasubordinatedboundsaux1} is bounded from below by the second summand on the right-hand side of~\eqref{eq:heatkernelalphasubordinatedboundsaux1}. On the other hand, by reflecting $\tau\mapsto\tau^{-1}$, the first summand on the right-hand side of~\eqref{eq:heatkernelalphasubordinatedboundsaux1} can be estimated by
\begin{align}
  \int_1^{\infty}d\tau\, \tau^{\frac{2\zeta-1+\alpha}{2}} \me{-\tau^{c_1}}
  \lesssim 1
\end{align}
from above. This gives the desired estimate. Now, we consider the second summand in \eqref{eq:heatkernelalphasubordinatedbounds}. For a lower bound, we drop it (like in the discussion of the lower bound for~\eqref{eq:heatkernelalphasubordinatedboundsaux1}). For an upper bound, we estimate
\begin{align}
  \label{eq:heatkernelalphasubordinatedboundsaux2}
  \begin{split}
    0 & \leq \int_{0}^{rs} \frac{d\tau}{\tau} \tau^{-\frac{1+\alpha}{2}} (rs)^{-\zeta}\, \me{-c((r-s)^2/\tau + 1/\tau^{c_1})}
        \lesssim (rs)^{-\zeta} \int_0^{rs} \frac{d\tau}{\tau} \tau^{-\frac{1+\alpha}{2}}\me{-c\tau^{-c_1}}\\
      & = (rs)^{-\zeta} \int_{(rs)^{-1}}^\infty \frac{d\tau}{\tau}\, \tau^{\frac{1+\alpha}{2}}\me{-c\tau^{c_1}}
        \asymp (rs)^{-\zeta} \cdot  (rs)^{-(\frac{\alpha+1}{2}-c_1)}\me{-c/(rs)^{c_1}}
        \lesssim 1,
  \end{split}
\end{align}
as desired. This concludes the case of $rs\vee(r-s)^2<1$.

\subsubsection{Case $(r-s)^2<1<rs$}

We show that $p_\zeta^{(\alpha)}(1,r,s)\sim(rs)^{-\zeta}$. We consider the first summand in \eqref{eq:heatkernelalphasubordinatedbounds}. For a lower bound we drop it, while for an upper bound we estimate
\begin{align}
  \label{eq:heatkernelalphasubordinatedboundsaux3}
  \begin{split}
    0 & \leq \int_{rs}^\infty \frac{d\tau}{\tau} \tau^{-\frac{2\zeta+1+\alpha}{2}} \me{-c(\frac{(r-s)^2}{\tau}+\tau^{-c_1})}
    \lesssim \int_{rs}^\infty \frac{d\tau}{\tau} \tau^{-\frac{2\zeta+1+\alpha}{2}} \\
      & \sim (rs)^{-\frac{2\zeta+1+\alpha}{2}}  \lesssim (rs)^{-\zeta}.
  \end{split}
\end{align}
We now consider the second summand in \eqref{eq:heatkernelalphasubordinatedbounds} and obtain
\begin{align}
  \label{eq:heatkernelalphasubordinatedboundsaux4}
  \begin{split}
    & \int_{0}^{rs} \frac{d\tau}{\tau} \tau^{-\frac{1+\alpha}{2}} (rs)^{-\zeta}\, \me{-c((r-s)^2/\tau+1/\tau^{c_1})} \\
    & \quad \asymp (rs)^{-\zeta} \left[\int_0^1 \frac{d\tau}{\tau} \tau^{-\frac{1+\alpha}{2}}\me{-c((r-s)^2/\tau+\tau^{-c_1})} + \int_1^{rs} \frac{d\tau}{\tau^{(3+\alpha)/2}}\right]
    \sim (rs)^{-\zeta}.
  \end{split}
\end{align}
This concludes the case of $(r-s)^2<1<rs$.

\subsubsection{Case $rs<1<(r-s)^2$}

We show that $p_\zeta^{(\alpha)}(1,r,s)\sim |r-s|^{-(2\zeta+1+\alpha)}$.
We first consider the second summand in \eqref{eq:heatkernelalphasubordinatedbounds}. For a lower bound we drop it, while for an upper bound we estimate
\begin{align}
  \label{eq:heatkernelalphasubordinatedboundsaux6}
  \begin{split}
    0 & \leq \int_{0}^{rs} \frac{d\tau}{\tau} \tau^{-\frac{1+\alpha}{2}} (rs)^{-\zeta}\, \me{-c(\frac{(r-s)^2}{\tau}+\frac{1}{\tau^{c_1}})}
    \lesssim (rs)^{-\zeta}\int_0^{rs} \frac{d\tau}{\tau} \tau^{-\frac{1+\alpha}{2}}\me{-c\frac{(r-s)^2}{\tau}} \\
    & = (rs)^{-\zeta} |r-s|^{-(1+\alpha)} \int_{(r-s)^2/(rs)}^\infty d\tau\, \tau^{\frac{\alpha-1}{2}} \me{-c\tau} \\
    & \asymp (rs)^{-\zeta} |r-s|^{-(1+\alpha)} \cdot \frac{|r-s|^{\alpha-1}}{(rs)^{(\alpha-1)/2}}\me{-c(r-s)^2/(rs)} \\
    & = |r-s|^{-(2\zeta+1+\alpha)} \cdot \frac{|r-s|^{2\zeta-1+\alpha}}{(rs)^{\frac{2\zeta-1+\alpha}{2}}}\me{-c\frac{(r-s)^2}{rs}}
    \lesssim |r-s|^{-(2\zeta+1+\alpha)}.
  \end{split}
\end{align}
To bound the first summand in \eqref{eq:heatkernelalphasubordinatedbounds}, we distinguish between $c_1\leq1$ and $c_1>1$. Consider first $c_1\leq1$. Then we have $\exp\left(-c((r-s)^2/\tau+1/\tau^{c_1})\right)\asymp\exp(-c(r-s)^2/\tau)$. In this case, we have the two-sided estimate
\begin{align}
  \label{eq:heatkernelalphasubordinatedboundsaux5}
  \begin{split}
    & \int_{rs}^\infty \frac{d\tau}{\tau} \tau^{-\zeta-\frac{1+\alpha}{2}} \me{-c(\frac{(r-s)^2}{\tau}+\frac{1}{\tau^{c_1}})}
      \asymp \int_{rs}^\infty \frac{d\tau}{\tau} \tau^{-\frac{2\zeta+1+\alpha}{2}} \me{-c\frac{(r-s)^2}{\tau}} \\
    & \quad \asymp \int_{rs}^{(r-s)^2} \frac{d\tau}{\tau} \tau^{-\frac{2\zeta+1+\alpha}{2}} \me{-c(r-s)^2/\tau} + \int_{(r-s)^2}^\infty \frac{d\tau}{\tau} \tau^{-\frac{2\zeta+1+\alpha}{2}} \\
    & \quad \asymp \frac{1}{|r-s|^{2\zeta+1+\alpha}} \left[\int_{rs/(r-s)^2}^1 \frac{d\tau}{\tau}\tau^{-\frac{2\zeta+1+\alpha}{2}}\me{-c/\tau} + 1 \right]
      \sim \frac{1}{|r-s|^{2\zeta+1+\alpha}}.
  \end{split}
\end{align}
In particular, this estimate suffices for an upper bound for all $c_1>0$ since $\me{-c/\tau^{c_1}}\leq1$. Thus, it remains to prove the lower bound when $c_1>1$, i.e., to find lower bounds for
\begin{align}
  \label{eq:heatkernelalphasubordinatedboundsaux5sub1}
  \int_0^\infty \frac{d\tau}{\tau}\, \tau^{-\frac{2\zeta+1+\alpha}{2}} \left[\! \exp\left(-c\frac{(r-s)^2}{\tau}\right)\one_{\tau>|r-s|^{2/(c_1-1)}} + \me{-c/\tau^{c_1}}\one_{\tau\in[rs,|r-s|^{2/(c_1-1)}]} \!\right]\!.
\end{align}
To this end, we distinguish $2/(c_1-1)\lessgtr2$. When $2/(c_1-1)\leq2$, we drop the second summand in \eqref{eq:heatkernelalphasubordinatedboundsaux5sub1} for a lower bound, while the first summand we treat as in the case $c_1\leq1$ in \eqref{eq:heatkernelalphasubordinatedboundsaux5}, by splitting the integral at $\tau=(r-s)^2$. The integral for $\tau<(r-s)^2$ can be dropped for a lower bound, while the integral for $\tau>(r-s)^2$ gives the desired contribution. Now consider $2/(c_1-1)>2$. In this case, we drop the first summand in \eqref{eq:heatkernelalphasubordinatedboundsaux5sub1} for a lower bound. On the other hand, we split the $\tau$-integration in the second summand at $\tau=(r-s)^2$, drop the contribution for $\tau<(r-s)^2$ for a lower bound, and estimate
\begin{align}
  \left[\int_{rs}^{(r-s)^2} \frac{d\tau}{\tau}\,\tau^{-\frac{2\zeta+1+\alpha}{2}}\me{-c/\tau^{c_1}} + \int_{(r-s)^2}^{(r-s)^{2/(c_1-1)}} \frac{d\tau}{\tau}\,\tau^{-\frac{2\zeta+1+\alpha}{2}} \right]
  \gtrsim \frac{1}{|r-s|^{2\zeta+1+\alpha}}.
\end{align}
This concludes the analysis of the case $rs<1<(r-s)^2$.

\subsubsection{Case $rs\wedge(r-s)^2>1$}

We show that 
\begin{align*}
  p_\zeta^{(\alpha)}(1,r,s) \sim |r-s|^{-(2\zeta+1+\alpha)} \one_{rs<(r-s)^2} + (rs)^{-\zeta}|r-s|^{-(1+\alpha)}\one_{rs>(r-s)^2}.
\end{align*}
Consider first $c_1\leq1$. Then, as before, we have $\exp\left(-c((r-s)^2/\tau+1/\tau^{c_1})\right)\asymp\exp(-c(r-s)^2/\tau)$. We consider the first summand in~\eqref{eq:heatkernelalphasubordinatedbounds} and estimate
\begin{align}
  \label{eq:heatkernelalphasubordinatedboundsaux7}
  \begin{split}
    & \int_{rs}^\infty \frac{d\tau}{\tau} \tau^{-\frac{2\zeta+1+\alpha}{2}} \me{-c((r-s)^2/\tau+1/\tau^{c_1})}
    \asymp \int_{rs}^\infty \frac{d\tau}{\tau} \tau^{-\frac{2\zeta+1+\alpha}{2}} \me{-c(r-s)^2/\tau} \\
    & \quad \asymp \one_{rs>(r-s)^2} \int_{rs}^\infty \frac{d\tau}{\tau} \tau^{-\frac{2\zeta+1+\alpha}{2}} \\
    & \qquad + \one_{rs<(r-s)^2} \left[\int_{rs}^{(r-s)^2}\frac{d\tau}{\tau} \tau^{-\frac{2\zeta+1+\alpha}{2}} \me{-c(r-s)^2/\tau} + \int_{(r-s)^2}^\infty \frac{d\tau}{\tau} \tau^{-\frac{2\zeta+1+\alpha}{2}}\right] \\
    & \quad \sim (rs)^{-\frac{2\zeta+1+\alpha}{2}}\one_{rs>(r-s)^2} + \frac{1}{|r-s|^{2\zeta+1+\alpha}}\one_{rs<(r-s)^2}.
  \end{split}
\end{align}
The second summand in this estimate is already in the desired form, while the first summand can be dropped for a lower bound and estimated from above by $(rs)^{-\zeta}|r-s|^{-1-\alpha}$ for the desired upper bound. We now consider the second summand in \eqref{eq:heatkernelalphasubordinatedbounds} and obtain
\begin{align}
  \label{eq:heatkernelalphasubordinatedboundsaux8}
  \begin{split}
    & (rs)^{-\zeta}\int_{0}^{rs} \frac{d\tau}{\tau} \tau^{-\frac{1+\alpha}{2}} \me{-c((r-s)^2/\tau+1/\tau^{c_1})}
    \sim (rs)^{-\zeta} \int_0^{rs} \frac{d\tau}{\tau} \tau^{-\frac{1+\alpha}{2}}\me{-(r-s)^2/\tau} \\
    & \quad \asymp \one_{rs>(r-s)^2}(rs)^{-\zeta}\left[\int_0^{(r-s)^2}\frac{d\tau}{\tau} \tau^{-\frac{1+\alpha}{2}} \me{-c(r-s)^2/\tau} + \int_{(r-s)^2}^{rs}\frac{d\tau}{\tau^{(3+\alpha)/2}}\right] \\
    & \qquad + \one_{rs<(r-s)^2}(rs)^{-\zeta}\int_0^{rs} \frac{d\tau}{\tau} \tau^{-\frac{1+\alpha}{2}} \me{-c(r-s)^2/\tau}\\
    & \quad \asymp \one_{rs>(r-s)^2}(rs)^{-\zeta}\left[\int_{(r-s)^{-2}}^\infty d\tau\, \tau^{\frac{\alpha-1}{2}} \me{-c\tau(r-s)^2} + |r-s|^{-(1+\alpha)}\right] \\
    & \qquad + \one_{rs<(r-s)^2}(rs)^{-\zeta} |r-s|^{-(1+\alpha)} \int_0^{rs/(r-s)^2} \frac{d\tau}{\tau} \tau^{-\frac{1+\alpha}{2}}\me{-c/\tau} \\
    & \quad \asymp (rs)^{-\zeta}\cdot |r-s|^{-(1+\alpha)} \left[\one_{rs>(r-s)^2} + \frac{|r-s|^{\alpha-1}}{(rs)^{(\alpha-1)/2}}\exp\left(-c\frac{(r-s)^2}{rs}\right)\one_{rs<(r-s)^2}\right].
  \end{split}
\end{align}
Comparing \eqref{eq:heatkernelalphasubordinatedboundsaux7} and \eqref{eq:heatkernelalphasubordinatedboundsaux8} for $rs>(r-s)^2$ shows that the term $(rs)^{-\zeta}\cdot|r-s|^{-(1+\alpha)}$ on the right-hand side of \eqref{eq:heatkernelalphasubordinatedboundsaux8} dominates the term $(rs)^{-(2\zeta+1+\alpha)/2}$ on the right-hand side of \eqref{eq:heatkernelalphasubordinatedboundsaux7} since $(rs)^{-\zeta}\cdot|r-s|^{-(1+\alpha)}\gtrsim (rs)^{-\frac{2\zeta+1+\alpha}{2}}$. On the other hand, if $rs<(r-s)^2$, then the $|r-s|^{-(2\zeta+1+\alpha)}$ term on the right-hand side of \eqref{eq:heatkernelalphasubordinatedboundsaux7} dominates the prefactor on the right-hand side of \eqref{eq:heatkernelalphasubordinatedboundsaux8} since
\begin{align}
  \begin{split}
    & \frac{|r-s|^{-2}}{(rs)^{\frac{2\zeta-1+\alpha}{2}}}\cdot\exp\left(-c\frac{(r-s)^2}{rs}\right)\one_{rs<(r-s)^2} \\
    & \quad = |r-s|^{-(2\zeta+1+\alpha)} \cdot \frac{|r-s|^{2\zeta-1+\alpha}}{(rs)^{\frac{2\zeta-1+\alpha}{2}}} \exp\left(-c\frac{(r-s)^2}{rs}\right)\one_{rs<(r-s)^2} \\
    & \quad \lesssim |r-s|^{-(2\zeta+1+\alpha)}.
  \end{split}
\end{align}
Thus, for $c_1\leq1$,
\begin{align}
  \begin{split}
    \eqref{eq:heatkernelalphasubordinatedbounds}
    & \sim (rs)^{-\zeta}|r-s|^{-1-\alpha}\one_{rs>(r-s)^2} + |r-s|^{-(2\zeta+1+\alpha)}\one_{rs<(r-s)^2},
  \end{split}
\end{align}
as needed. Now suppose $c_1>1$. Then the previous analysis suffices for the upper bound since $\me{-c/\tau^{c_1}}\leq1$. Thus, it suffices to prove the lower bound when $c_1>1$, i.e., it remains to find lower bounds for
\begin{align}
  \label{eq:heatkernelalphasubordinatedboundsaux7sub1}
  \begin{split}
    & \int_0^\infty \frac{d\tau}{\tau}\, \left[ \tau^{-\frac{2\zeta+1+\alpha}{2}} \left[\me{-c\frac{(r-s)^2}{\tau}}\one_{\tau>|r-s|^{2/(c_1-1)}\vee rs} + \me{-c/\tau^{c_1}}\one_{\tau\in[rs,|r-s|^{2/(c_1-1)}]}\right] \right. \\
    & \quad \left. + (rs)^{-\zeta} \tau^{-\frac{1+\alpha}{2}}\left[\me{-c\frac{(r-s)^2}{\tau}}\one_{|r-s|^{2/(c_1-1)}<\tau<rs} + \me{-c/\tau^{c_1}}\one_{\tau<rs\wedge |r-s|^{2/(c_1-1)}}\right]\right].
  \end{split}
\end{align}
When $rs<(r-s)^2$, we can argue exactly as in the previously analyzed case $rs<1<(r-s)^2$. Thus, we assume $(r-s)^2<rs$. In this case, we drop the first line in~\eqref{eq:heatkernelalphasubordinatedboundsaux7sub1} for a lower bound, i.e., it remains to control the second line. If $2/(c_1-1)\leq2$ then we drop the second summand in the second line of \eqref{eq:heatkernelalphasubordinatedboundsaux7sub1} and estimate the first summand there precisely as in \eqref{eq:heatkernelalphasubordinatedboundsaux8} by splitting the $\tau$-integral at $\tau=(r-s)^2$ and dropping the integral over $\{\tau<(r-s)^2\}$. On the other hand, if $2/(c_1-1)>2$, then we drop the first summand in the second line of \eqref{eq:heatkernelalphasubordinatedboundsaux7sub1}. The second summand is bounded from below by
\begin{align}
  & (rs)^{-\zeta} \int_0^\infty \frac{d\tau}{\tau}\, \tau^{-\frac{1+\alpha}{2}} \me{-c/\tau^{c_1}}\one_{(r-s)^2<\tau<rs\wedge |r-s|^{2/(c_1-1)}}
    \gtrsim (rs)^{-\zeta} |r-s|^{-(1+\alpha)}.
\end{align}
This concludes the analysis of the case $rs\wedge(r-s)^2>1$ and thereby the proof of the bounds in \eqref{eq:heatkernelalpha1subordinatedboundsfinalaux1}.

\subsection{Proof of \eqref{eq:heatkernelalpha1weightedsubordinatedboundsfinal}}\label{A.2}

The bounds \eqref{eq:heatkernelalpha1subordinatedboundsfinalaux1} imply
\begin{align}
  \begin{split}
    \label{eq:heatkernelalpha1subordinatedboundsfinalaux2}
    p_\zeta^{(\alpha)}(1,r,s)
    & \sim_{\zeta,\alpha} \left(1+rs\right)^{\zeta}\one_{(r-s)^2<1} + |r-s|^{-(2\zeta+1+\alpha)}\one_{rs<1<(r-s)^2} \\
    & \quad + \left(|r-s|^{-2\zeta}\one_{1<rs<(r-s)^2}+ (rs)^{-\zeta}\one_{rs>(r-s)^2>1}\right)|r-s|^{-(1+\alpha)}.
  \end{split}
\end{align}
Using \eqref{eq:heatkernelalpha1subordinatedboundsfinalaux2}, we now show the desired bound \eqref{eq:heatkernelalpha1weightedsubordinatedboundsfinal}, i.e.,
\begin{align*}
  \begin{split}
    \label{eq:heatkernelalpha1subordinatedboundsfinal}
    p_\zeta^{(\alpha)}(1,r,s)
    & \sim_{\zeta,\alpha} \frac{1}{|r-s|^{1+\alpha}\cdot(r+s)^{2\zeta} + (1+r+s)^{2\zeta}}
  \end{split}
\end{align*}
for all $r,s>0$. To that end, we distinguish the cases $(r-s)^2\lessgtr1$.
\smallskip\\
(1) If $(r-s)^2<1$, then, by $2rs=r^2+s^2-(r-s)^2$, we have $p_\zeta^{(\alpha)}(1,r,s)\sim (1+rs)^{-\zeta}\sim(1+r+s)^{-2\zeta}$. This gives the claimed estimate for $p_\zeta^{(\alpha)}(1,r,s)$. Note that we use $|r-s|^{1+\alpha}(r+s)^{2\zeta}\leq(r+s)^{2\zeta+1+\alpha}\lesssim1$ for $r+s\lesssim1$ and $\zeta>-1/2$.
\\
(2) If $(r-s)^2>1$, then $|r-s|^{1+\alpha}(r+s)^{2\zeta}\gtrsim(1+r+s)^{2\zeta}$. We now distinguish the cases $rs\lessgtr(r-s)^2$ and use $(r+s)^2=(r-s)^2+4rs$. If $(r-s)^2>rs$, then the claimed estimate for $p_\zeta^{(\alpha)}(1,r,s)$ follows immediately. If $(r-s)^2<rs$, then $|r-s|^{1+\alpha}(rs)^{\zeta} \sim |r-s|^{1+\alpha}(r+s)^{2\zeta}$ as desired.



\def\cprime{$'$}

\end{document}